\documentclass[reqno,a4paper]{amsart}
\usepackage{setspace,amssymb}
\usepackage{ifpdf}
\ifpdf
 \usepackage[hyperindex]{hyperref}
\else
 \expandafter\ifx\csname dvipdfm\endcsname\relax
 \usepackage[hypertex,hyperindex]{hyperref}
 \else
 \usepackage[dvipdfm,hyperindex]{hyperref}
 \fi
\fi
\allowdisplaybreaks[4]
\numberwithin{equation}{section}
\theoremstyle{plain}
\newtheorem{thm}{Theorem}[section]

\newtheorem{cor}{Corollary}[section]
\newtheorem{conj}{Conjecture}[section]
\theoremstyle{plain}

\theoremstyle{remark}
\newtheorem{rem}{Remark}[section]
\DeclareMathOperator{\td}{d}
\newcommand{\bell}{\textup{B}}

\begin{document}

\title[Recurrence relations and inequalities for Stirling numbers]
{Diagonal recurrence relations, inequalities, and monotonicity related to Stirling numbers}

\author[F. Qi]{Feng Qi}
\address{Department of Mathematics, College of Science, Tianjin Polytechnic University, Tianjin City, 300387, China; Institute of Mathematics, Henan Polytechnic University, Jiaozuo City, Henan Province, 454010, China}
\email{\href{mailto: F. Qi <qifeng618@gmail.com>}{qifeng618@gmail.com}, \href{mailto: F. Qi <qifeng618@hotmail.com>}{qifeng618@hotmail.com}, \href{mailto: F. Qi <qifeng618@qq.com>}{qifeng618@qq.com}}
\urladdr{\url{http://qifeng618.wordpress.com}}

\begin{abstract}
In the paper, the author derives several ``diagonal'' recurrence relations, constructs some inequalities, finds monotonicity, and poses a conjecture related to Stirling numbers of the second kind.
\end{abstract}



\thanks{This paper was typeset using \AmS-\LaTeX}

\maketitle

\section{Introduction}

In mathematics, Stirling numbers arise in a variety of combinatorics problems. They are introduced in the eighteen century by James Stirling. There are two kinds of Stirling numbers: Stirling numbers of the first and second kinds. Some properties and recurrence relations of Stirling numbers of these two kinds are collected in, for example,~\cite[Chapter~V]{Comtet-Combinatorics-74}.
\par
Some Stirling number of the second kind $S(n,k)$ is the number of ways of partitioning a set of $n$ elements into $k$ nonempty subsets.
It may be computed by
\begin{equation}\label{S(n-k)-explicit}
S(n,k)=\frac1{k!}\sum_{i=0}^k(-1)^i\binom{k}{i}(k-i)^n
\end{equation}
and may be generated by
\begin{equation}\label{2stirling-gen-funct-exp}
\frac{(e^x-1)^k}{k!}=\sum_{n=k}^\infty S(n,k)\frac{x^n}{n!}, \quad k\in\{0\}\cup\mathbb{N}.
\end{equation}
\par
In this paper, we will derive several ``diagonal'' recurrence relations, construct some inequalities, and find monotonicity related to $S(n,k)$. By the way, we will also pose a conjecture on monotonicity and logarithmic concavity of sequences related to Stirling numbers of the second kind $S(n,k)$.

\section{Several ``diagonal'' recurrence relations of $S(n,k)$}

In~\cite[p.~209]{Comtet-Combinatorics-74}, two ``vertical'' and two ``horizontal'' recurrence relations for $S(n,k)$ were listed. Relative to the words ``vertical'' and ``horizontal'', we may call the following formulas~\eqref{StirlingS2-Bell-eq<2k} and~\eqref{StirlingS2-Bell-eq-simp} the ``diagonal'' recurrence relations for Stirling numbers of the second kind $S(n,k)$.

\begin{thm}\label{StirlingS2-Bell-thm}
For $n>k\ge0$, we have
\begin{align}\label{StirlingS2-Bell-eq<2k}
S(n,k)&=\binom{n}{k}\sum_{\ell=1}^{n-k}(-1)^{\ell}\frac{\binom{k}{\ell}} {\binom{n-k+\ell}{n-k}}\sum_{i=0}^\ell(-1)^{i} \binom{n-k+\ell}{\ell-i}S(n-k+i,i)\\ \label{StirlingS2-Bell-eq-simp}
&=(-1)^n\sum^{k-1}_{i=2k-n}(-1)^{i} \binom{n}{i} \binom{i-1}{2 k-n-1}S(n-i,k-i),
\end{align}
where the conventions that
\begin{equation}\label{binomial-convention-eq}
\binom00=1,\quad \binom{-1}{-1}=1, \quad \text{and}\quad \binom{p}q=0
\end{equation}
for $p\ge0>q$ are adopted.
\end{thm}

\begin{proof}
The equation~\eqref{2stirling-gen-funct-exp} may be rearranged as
\begin{equation}\label{2stirling-gen-rew}
\biggl(\frac{e^x-1}x\biggr)^k=\sum_{n=0}^\infty \frac{S(n+k,k)}{\binom{n+k}{k}}\frac{x^n}{n!}, \quad k\in\{0\}\cup\mathbb{N}.
\end{equation}
Consequently, as coefficients of the power series expansion of the function $\bigl(\frac{e^x-1}x\bigr)^k$,  \begin{equation*}
\frac{S(n+k,k)}{\binom{n+k}{k}}=\lim_{x\to0}\frac{\td^n}{\td x^n} \biggl[\biggl(\frac{e^x-1}x\biggr)^k\biggr],
\end{equation*}
that is, for $n\ge k\ge0$,
\begin{equation}\label{StirlingS2-IDL-eq}
S(n,k)=\binom{n}{k}\lim_{x\to0}\frac{\td^{n-k}}{\td x^{n-k}}\biggl[\biggl(\int_1^eu^{x-1}\td u\biggr)^k\biggr].
\end{equation}
\par
In combinatorics, Bell polynomials of the second kind, or say, the partial Bell polynomials, denoted by $\bell_{n,k}(x_1,x_2,\dotsc,x_{n-k+1})$, may be defined by
\begin{equation}\label{}
\bell_{n,k}(x_1,x_2,\dotsc,x_{n-k+1})=\sum_{\substack{1\le i\le n,\ell_i\in\{0\}\cup\mathbb{N}\\ \sum_{i=1}^ni\ell_i=n\\
\sum_{i=1}^n\ell_i=k}}\frac{n!}{\prod_{i=1}^{n-k+1}\ell_i!} \prod_{i=1}^{n-k+1}\biggl(\frac{x_i}{i!}\biggr)^{\ell_i}
\end{equation}
for $n\ge k\ge0$, and the well-known Fa\`a di Bruno formula may be described in terms of Bell polynomials of the second kind $\bell_{n,k}(x_1,x_2,\dotsc,x_{n-k+1})$ by
\begin{equation}\label{Bruno-Bell-Polynomial}
\frac{\td^n}{\td x^n}f\circ g(x)=\sum_{k=1}^nf^{(k)}(g(x)) \bell_{n,k}\bigl(g'(x),g''(x),\dotsc,g^{(n-k+1)}(x)\bigr).
\end{equation}
See~\cite[p.~134, Theorem~A]{Comtet-Combinatorics-74} and~\cite[p.~139, Theorem~C]{Comtet-Combinatorics-74}.
In~\cite[Theorem~1]{Guo-Qi-JANT-Bernoulli.tex} and~\cite[Example~4.2]{Zhang-Yang-Oxford-Taiwan-12}, it was derived that
\begin{equation}\label{B-S-frac-value}
\bell_{n,k}\biggl(\frac12, \frac13,\dotsc,\frac1{n-k+2}\biggr)
=\frac{n!}{(n+k)!}\sum_{i=0}^k(-1)^{k-i}\binom{n+k}{k-i}S(n+i,i).
\end{equation}
See also~\cite{Special-Bell2Euler.tex} and closely related references therein.
Consequently, we may obtain the following conclusions:
\begin{enumerate}
\item
When $1\le m\le k$,
\begin{multline}\label{m-le-k-eq}\raisetag{9em}
\frac{\td^m}{\td x^m}\biggl[\biggl(\int_1^eu^{x-1}\td u\biggr)^k\biggr]
=\sum_{\ell=1}^m\frac{k!}{(k-\ell)!}\biggl(\int_1^eu^{x-1}\td u\biggr)^{k-\ell}\\
\begin{aligned}
&\quad\times\bell_{m,\ell}\biggl(\int_1^eu^{x-1}\ln u\td u,\int_1^eu^{x-1}(\ln u)^2\td u,\dotsc, \int_1^eu^{x-1}(\ln u)^{m-\ell+1}\td u\biggr)\\
&\to\sum_{\ell=1}^m\frac{k!}{(k-\ell)!}\bell_{m,\ell}\biggl(\frac12,\frac13,\dotsc, \frac1{m-\ell+2}\biggr), \quad x\to0\\
&=\sum_{\ell=1}^m\frac{\binom{k}{\ell}}{\binom{m+\ell}{m}} \sum_{i=0}^\ell(-1)^{i}\binom{m+\ell}{i}S(m+\ell-i,\ell-i)\\
&=\sum_{\ell=1}^m(-1)^{\ell}\frac{\binom{k}{\ell}}{\binom{m+\ell}{m}}\sum_{i=0}^\ell(-1)^{i} \binom{m+\ell}{m+i}S(m+i,i);
\end{aligned}
\end{multline}
\item
Similarly, when $m>k$, we have
\begin{equation}\label{m>k-eq}
\frac{\td^m}{\td x^m}\biggl[\biggl(\int_1^eu^{x-1}\td u\biggr)^k\biggr]
\to\sum_{\ell=1}^k(-1)^{\ell}\frac{\binom{k}{\ell}}{\binom{m+\ell}{m}}\sum_{i=0}^\ell(-1)^{i} \binom{m+\ell}{m+i}S(m+i,i)
\end{equation}
as $x\to0$.
\end{enumerate}
Since the convention that $\binom{k}{m}=0$ for $m>k$, the equation~\eqref{m-le-k-eq} holds for all $m\ge1$ and includes~\eqref{m>k-eq}. Substituting~\eqref{m-le-k-eq} into~\eqref{StirlingS2-IDL-eq} produces
\begin{equation*}
S(n,k)=\binom{n}{k}\sum_{\ell=1}^{n-k}(-1)^{\ell}\frac{\binom{k}{\ell}} {\binom{n-k+\ell}{n-k}}\sum_{i=0}^\ell(-1)^{i} \binom{n-k+\ell}{n-k+i}S(n-k+i,i)
\end{equation*}
for $n-k\ge1$. The formula~\eqref{StirlingS2-Bell-eq<2k} follows.
\par
Interchanging two sums in~\eqref{StirlingS2-Bell-eq<2k} and then computing the inner sum result in
\begin{align*}
S(n,k)&=\binom{n}{k}\sum_{i=1}^{n-k}(-1)^{i} \Biggl[\sum_{\ell=i}^{n-k}(-1)^{\ell}\frac{\binom{k}{\ell}} {\binom{n-k+\ell}{n-k}}\binom{n-k+\ell}{\ell-i}\Biggr] S(n-k+i,i)\\
&=\sum_{i=1}^{n-k}(-1)^{i}\binom{n}{k-i} \Biggl[\sum_{\ell=i}^{n-k}(-1)^{\ell} \binom{k-i}{k-\ell}\Biggr] S(n-k+i,i)\\
&=
\begin{cases}\displaystyle
\sum_{i=1}^{n-k}(-1)^{i}\binom{n}{k-i} \Biggl[\sum_{\ell=i}^{n-k}(-1)^{\ell} \binom{k-i}{k-\ell}\Biggr] S(n-k+i,i), & k<n\le 2k\\\displaystyle
\sum_{i=1}^k(-1)^{i}\binom{n}{k-i} \Biggl[\sum_{\ell=i}^{k}(-1)^{\ell} \binom{k-i}{k-\ell}\Biggr] S(n-k+i,i), & n>2k
\end{cases}\\
&=
\begin{cases}\displaystyle
\sum^{k-1}_{i=2k-n}\binom{n}{i} \Biggl[\sum^{i-(2k-n)}_{\ell=0}(-1)^{\ell} \binom{i}{\ell}\Biggr] S(n-i,k-i), & k<n\le 2k\\\displaystyle
\sum_{i=0}^{k-1}\binom{n}{i}\Biggl[\sum_{\ell=0}^{i}(-1)^\ell\binom{i}{\ell}\Biggr] S(n-i,k-i), & n>2k
\end{cases}\\
&=\sum^{k-1}_{i=2k-n}\binom{n}{i} \Biggl[\sum^{i-(2k-n)}_{\ell=0}(-1)^{\ell} \binom{i}{\ell}\Biggr] S(n-i,k-i)\\
&=(-1)^{n-2k}\sum^{k-1}_{i=2k-n}(-1)^{i} \binom{n}{i} \binom{i-1}{i-(2 k-n)}S(n-i,k-i).
\end{align*}
The formula~\eqref{StirlingS2-Bell-eq-simp} follows.
The proof of Theorem~\ref{StirlingS2-Bell-thm} is complete.
\end{proof}

\begin{rem}
In~\cite[Theorem~1.1]{notes-Stirl-No-JNT-rev.tex}, a ``diagonal'' recurrence relation
\begin{equation}\label{s(n-k)=s(n-k)-id}
s(n,k)=\sum_{m=1}^{n}\sum_{\ell=k-m}^{k-1}(-1)^{k+m-\ell} \binom{n}{\ell}\binom{\ell}{k-m} s(n-\ell,k-\ell)
\end{equation}
for $n\ge k\ge1$ was discovered for Stirling numbers of the first kind $s(n,k)$ which may be generated by
\begin{equation} \label{gen-funct-3}
\frac{[\ln(1+x)]^k}{k!}=\sum_{n=k}^\infty s(n,k)\frac{x^n}{n!},\quad |x|<1.
\end{equation}
\par
As done in the derivation of~\eqref{StirlingS2-Bell-eq-simp}, we may also interchange two sums in~\eqref{StirlingS2-Bell-eq<2k} and compute the inner sum as follows:
\begin{align*}
s(n,k)&=(-1)^k\sum_{\ell=k-n}^{k-1}(-1)^\ell\binom{n}{\ell}\Biggl[\sum_{m=k-\ell}^n(-1)^{m} \binom{\ell}{k-m}\Biggr] s(n-\ell,k-\ell)\\
&=(-1)^{n-k}\sum_{\ell=k-n}^{k-1}(-1)^\ell\binom{n}{\ell} \binom{\ell-1}{k-n-1}s(n-\ell,k-\ell), \quad n\ge k\ge1,
\end{align*}
that is,
\begin{equation}\label{1stirling-diagonal-eq}
s(n,k)=(-1)^{n-k}\sum_{\ell=0}^{k-1}(-1)^\ell\binom{n}{\ell} \binom{\ell-1}{k-n-1}s(n-\ell,k-\ell), \quad n\ge k\ge1,
\end{equation}
where the conventions listed in~\eqref{binomial-convention-eq} are also adopted. The recurrence relation~\eqref{1stirling-diagonal-eq} may also be called as a ``diagonal'' recurrence relation for Stirling numbers of the first kind $s(n,k)$.
\par
The relations~\eqref{s(n-k)=s(n-k)-id} and~\eqref{1stirling-diagonal-eq} are different from two ``vertical'' and two ``horizontal'' recurrence relations collected in~\cite[p.~215]{Comtet-Combinatorics-74}.
\end{rem}

\section{Inequalities and monotonicity related to $S(n,k)$}

After establishing and discussing ``diagonal'' recurrence relations for Stirling numbers of the fist and second kinds $S(n,k)$ and $s(n,k)$, we now construct and deduce, with the help of the formula~\ref{StirlingS2-IDL-eq} and in light of properties of absolutely monotonic functions, some inequalities and monotonicity related to Stirling numbers of the second kind $S(n,k)$.

\begin{thm}\label{S(n-k)-matrix-thm}
Let $m\ge1$ be a positive integer and let $|a_{ij}|_m$ denote a determinant of order $m$ with elements $a_{ij}$.
\begin{enumerate}
\item
If $a_i$ for $1\le i\le m$ are non-negative integers, then
\begin{equation}
\Biggl|\frac{S(a_i+a_j+k,k)}{\binom{a_i+a_j+k}{k}}\Biggr|_m \ge0
\end{equation}
and
\begin{equation}
\Biggl|(-1)^{a_i+a_j}\frac{S(a_i+a_j+k,k)}{\binom{a_i+a_j+k}{k}}\Biggr|_m\ge0
\end{equation}
hold true for all given $k\in\mathbb{N}$.
\item
Let $q=(q_1,q_2,\dotsc,q_n)$ be a real $n$-tuple of non-negative integers and let $a=(a_1,a_2,\dotsc,a_n)$ and $b=(b_1,b_2,\dots,b_n)$ be non-increasing $n$-tuples of non-negative integers such that $a\succeq_q b$, that is,
\begin{equation*}
\sum_{i=1}^kq_ia_i\ge\sum_{i=1}^kq_ib_i
\end{equation*}
for $1\le k\le n-1$ and
\begin{equation*}
\sum_{i=1}^nq_ia_i=\sum_{i=1}^nq_ib_i.
\end{equation*}
Then the inequality
\begin{equation}\label{S(n-k)-prod-ineq}
\prod_{i=1}^n\Biggl[\frac{S(a_i+k,k)}{\binom{a_i+k}{k}}\Biggr]^{q_i}
\ge\prod_{i=1}^n\Biggl[\frac{S(b_i+k,k)}{\binom{b_i+k}{k}}\Biggr]^{q_i}
\end{equation}
hods true for all given $k\in\mathbb{N}$.
\end{enumerate}
\end{thm}

\begin{proof}
A function $f$ is said to be absolutely monotonic on an interval $I$ if $f$ has derivatives of all
orders on $I$ and $0\le f^{(n)}(x)<\infty$ for $x\in I$ and $n\ge0$. See~\cite{absolute-mon-simp.tex} and Chapter~XIII in~\cite{mpf-1993}.
In~\cite{two-place} and~\cite[p.~367]{mpf-1993}, it was recited that if $f$ is an absolutely monotonic function on $[0,\infty)$, then
\begin{equation}\label{matrix-eq-1}
\bigl|f^{(a_i+a_j)}(x)\bigr|_m\ge0
\end{equation}
and
\begin{equation}\label{matrix-eq-2}
\bigl|(-1)^{a_i+a_j}f^{(a_i+a_j)}(x)\bigr|_m\ge0.
\end{equation}
In~\cite[p.~368]{mpf-1993} and~\cite[p.~429]{Remarks-on-Fink}, it was stated that if $f$ is an absolutely monotonic function on $[0,\infty)$ and $a\succeq_q b$, then
\begin{equation}\label{matrix-eq-3}
\prod_{i=1}^n\bigl[f^{(a_i)}(x)\bigr]^{q_i}
\ge\prod_{i=1}^n\bigl[f^{(b_i)}(x)\bigr]^{q_i}.
\end{equation}
\par
It is easy to see that
\begin{equation*}
\biggl(\frac{e^x-1}x\biggr)^{(m)}=\int_1^eu^{x-1}(\ln u)^m\td u, \quad m\ge0,
\end{equation*}
see~\cite{1st-Sirling-Number-2012.tex, Qi-Springer-2012-Srivastava.tex} and plenty of closely-related references cited therein.
This means that $\frac{e^x-1}x=\int_1^eu^{x-1}\td u$ is absolutely monotonic on $(-\infty,\infty)$. As a result, the function
\begin{equation}
H_k(x)=\biggl(\frac{e^x-1}x\biggr)^k=\biggl(\int_1^eu^{x-1}\td u\biggr)^k, \quad k\in\mathbb{N}
\end{equation}
is also absolutely monotonic on $(-\infty,\infty)$ and, by the formula~\ref{StirlingS2-IDL-eq},
\begin{equation}
\lim_{x\to0}H_k^{(\ell)}(x)=\frac{S(\ell+k,k)}{\binom{\ell+k}{k}}, \quad \ell\in\{0\}\cup\mathbb{N}.
\end{equation}
Making use of inequalities~\eqref{matrix-eq-1}, \eqref{matrix-eq-2}, and~\eqref{matrix-eq-3} and taking the limit $x\to0$ find that
\begin{align*}
0\le\bigl|H_k^{(a_i+a_j)}(x)\bigr|_m
&\to\Biggl|\frac{S(a_i+a_j+k,k)}{\binom{a_i+a_j+k}{k}}\Biggr|_m,\\
0\le\bigl|(-1)^{a_i+a_j}H_k^{(a_i+a_j)}(x)\bigr|_m
&\to\Biggl|(-1)^{a_i+a_j}\frac{S(a_i+a_j+k,k)}{\binom{a_i+a_j+k}{k}}\Biggr|_m,
\end{align*}
and
\begin{equation*}
\prod_{i=1}^n\Biggl[\frac{S(a_i+k,k)}{\binom{a_i+k}{k}}\Biggr]^{q_i}
\leftarrow\prod_{i=1}^n\bigl[H_k^{(a_i)}(x)\bigr]^{q_i}
\ge\prod_{i=1}^n\bigl[H_k^{(b_i)}(x)\bigr]^{q_i}
\to\prod_{i=1}^n\Biggl[\frac{S(b_i+k,k)}{\binom{b_i+k}{k}}\Biggr]^{q_i}.
\end{equation*}
The proof of Theorem~\ref{S(n-k)-matrix-thm} is complete.
\end{proof}

\begin{cor}\label{S(n-k)-convex-cor}
For any given $k\in\mathbb{N}$, the infinite sequence
\begin{equation}\label{S-frac-binom-seq}
\Biggl\{\frac{S(n+k,k)}{\binom{n+k}{k}}\Biggr\}_{n\ge0}
\end{equation}
is logarithmically convex with respect to $n$.
\end{cor}

\begin{proof}
Letting
\begin{equation*}
n=2,\quad q_1=q_2=1, \quad a_1=\ell+2, \quad a_2=\ell, \quad\text{and}\quad b_1=b_2=\ell+1
\end{equation*}
in the inequality~\eqref{S(n-k)-prod-ineq} leads to
\begin{equation}\label{S(n-k)-convex-ineq}
\frac{S(\ell+k+2,k)}{\binom{\ell+k+2}{k}} \frac{S(\ell+k,k)}{\binom{\ell+k}{k}}
\ge\Biggl[\frac{S(\ell+k+1,k)}{\binom{\ell+k+1}{k}}\Biggr]^2, \quad \ell\ge0.
\end{equation}
As a result, the sequence~\eqref{S-frac-binom-seq} is logarithmically convex.
The proof of Corollary~\ref{S(n-k)-convex-cor} is complete.
\end{proof}

\begin{rem}
The ideas employed in Theorem~\ref{S(n-k)-matrix-thm} and Corollary~\ref{S(n-k)-convex-cor} have also been applied in the articles~\cite{Norlund-No-CM-JNT.tex, 1st-Sirling-Number-2012.tex, Bernoulli2nd-Property.tex}.
\end{rem}

\section{Monotonicity}

After establishing diagonal recurrence relations and constructing inequalities related to Stirling numbers of the second kind $S(n,k)$, we are now in a position to create an infinite sequence in terms of Stirling numbers of the second kind $S(n,k)$ and to prove its increasing monotonicity.

\begin{thm}\label{S(n-k)-increasing-seq-thm}
For any fixed positive integers $n,k$ with $n\ge k\ge2$, let
\begin{equation}\label{frak-S(n-k)}
\mathfrak{S}_1(n,k)=S^2(n,k-1)-S(n,k-2)S(n,k).
\end{equation}
Then the infinite sequence $\{\mathfrak{S}_1(n+m,k+m)\}_{m\ge0}$ is strictly increasing with respect to $m$.
\end{thm}

\begin{proof}
It is well known in combinatorics that Stirling numbers of the second kind $S(n,k)$ satisfy $S(0,0)=1$, $S(n,0)=S(0,k)=0$ for $n,k\ge1$, and the ``triangular'' recurrence relation
\begin{equation}\label{Stirling2-recur-eq}
S(n,k)=kS(n-1,k)+S(n-1,k-1)
\end{equation}
for $n\ge k\ge1$. See~\cite[p.~208]{Comtet-Combinatorics-74}. Hence, the inequality~\eqref{S(n-k)-convex-ineq} may be rearranged as
\begin{gather*}
\frac{S(i+k+2,k)}{\binom{i+k+2}{k}} \frac{S(i+k,k)}{\binom{i+k}{k}}
=\frac{S(i+k+1,k-1)+kS(i+k+1,k)}{\binom{i+k+2}{k}} \frac{S(i+k,k)}{\binom{i+k}{k}}\\
=\frac{S(i+k,k-2)+(2k-1)S(i+k,k-1)+k^2S(i+k,k)}{\binom{i+k+2}{k}} \frac{S(i+k,k)}{\binom{i+k}{k}}\\
\ge\Biggl[\frac{S(i+k+1,k)}{\binom{i+k+1}{k}}\Biggr]^2
=\Biggl[\frac{S(i+k,k-1)+kS(i+k,k)}{\binom{i+k+1}{k}}\Biggr]^2.
\end{gather*}
Replacing $i+k$ by $n$ in the above inequality and simplifying give
\begin{equation*}
\frac{S(n,k-2)+(2k-1)S(n,k-1)+k^2S(n,k)}{\binom{n+2}{k}} \frac{S(n,k)}{\binom{n}{k}}
\ge\Biggl[\frac{S(n,k-1)+kS(n,k)}{\binom{n+1}{k}}\Biggr]^2,
\end{equation*}
that is, by the recurrence relation~\eqref{Stirling2-recur-eq},
\begin{gather*}
\begin{aligned}
&\quad\frac{S(n,k-2)S(n,k)}{\binom{n+2}{k}\binom{n}{k}}-\frac{S^2(n,k-1)}{\binom{n+1}{k}^2}\\
&\ge \Biggl[\frac1{\binom{n+1}{k}^2}-\frac1{\binom{n+2}{k}\binom{n}{k}}\Biggr]k^2S^2(n,k)
+\Biggl[\frac{2k}{\binom{n+1}{k}^2} -\frac{2k-1}{\binom{n+2}{k}\binom{n}{k}}\Biggr]S(n,k-1)S(n,k),
\end{aligned}\\
\begin{aligned}
&\quad\frac{S(n,k-2)S(n,k)-S^2(n,k-1)}{\binom{n+2}{k}\binom{n}{k}}
+\Biggl[\frac1{\binom{n+2}{k}\binom{n}{k}}-\frac1{\binom{n+1}{k}^2}\Biggr]S^2(n,k-1)\\
&\ge \Biggl[\frac1{\binom{n+1}{k}^2}-\frac1{\binom{n+2}{k}\binom{n}{k}}\Biggr]k^2S^2(n,k)
+\Biggl[\frac{2k}{\binom{n+1}{k}^2} -\frac{2k-1}{\binom{n+2}{k}\binom{n}{k}}\Biggr]S(n,k-1)S(n,k),
\end{aligned}\\
\begin{aligned}
\frac{\mathfrak{S}_1(n,k)}{\binom{n+2}{k}\binom{n}{k}}
&\le \Biggl[\frac{2k-1}{\binom{n+2}{k}\binom{n}{k}} -\frac{2k}{\binom{n+1}{k}^2}\Biggr]S(n,k-1)S(n,k)\\
&\quad+\Biggl[\frac1{\binom{n+2}{k}\binom{n}{k}}-\frac1{\binom{n+1}{k}^2}\Biggr] \bigl[k^2S^2(n,k)+S^2(n,k-1)\bigr],
\end{aligned}\\
\begin{aligned}
\mathfrak{S}_1(n,k)
&\le \Biggl[2k-1 -2k\frac{\binom{n+2}{k}\binom{n}{k}}{\binom{n+1}{k}^2}\Biggr]S(n,k-1)S(n,k)\\
&\quad+\Biggl[1-\frac{\binom{n+2}{k}\binom{n}{k}}{\binom{n+1}{k}^2}\Biggr] \bigl[k^2S^2(n,k)+S^2(n,k-1)\bigr],
\end{aligned}\\
\begin{aligned}
\mathfrak{S}_1(n,k)&\le\Biggl[1-\frac{\binom{n+2}{k}\binom{n}{k}} {\binom{n+1}{k}^2}\Biggr]\bigl[k^2S^2(n,k)
+2kS(n,k-1)S(n,k)+S^2(n,k-1)\bigr]\\
&\quad-S(n,k-1)S(n,k)\\
&=\Biggl[1-\frac{\binom{n+2}{k}\binom{n}{k}}{\binom{n+1}{k}^2}\Biggr] \bigl[kS(n,k)+S(n,k-1)\bigr]^2-S(n,k-1)S(n,k)\\
&=\frac{k}{(n+1) (n-k+2)}S^2(n+1,k)-S(n,k-1)S(n,k).
\end{aligned}
\end{gather*}
In order to prove the increasing monotonicity of the infinite sequence $\{\mathfrak{S}_1(n+m,k+m)\}_{m\ge0}$, it suffices to show
\begin{equation*}
\frac{kS^2(n+1,k)}{(n+1) (n-k+2)}-S(n,k-1)S(n,k)\le S^2(n+1,k)-S(n+1,k-1)S(n+1,k+1)
\end{equation*}
which may be reformulated as
\begin{equation}\label{suffice-eq-S(n-k)}
\frac{S(n+1,k-1)S(n+1,k+1)}{S^2(n+1,k)}-\frac{S(n,k-1)S(n,k)}{S^2(n+1,k)}
\le \frac{(n+2)(n-k+1)}{(n+1) (n-k+2)}.
\end{equation}
In~\cite[p.~698]{Sibuya-AISM-1988}, it was proved by the recurrence relation in~\eqref{Stirling2-recur-eq} and by induction that
\begin{equation}
(m-1)(n-m)S^2(n,m)>(m+1)(n-m+1)S(n,m-1)S(n,m+1)
\end{equation}
for $2\le m\le n-1$, which may be rewritten as
\begin{equation}\label{S-Ratio-ineq-upper}
\frac{S(n+1,k-1)S(n+1,k+1)}{S^2(n+1,k)}<\frac{(k-1)(n-k+1)}{(k+1)(n-k+2)}
\end{equation}
for $2\le k\le n$. Therefore, in order to show~\eqref{suffice-eq-S(n-k)}, it is sufficient to verify
\begin{align*}
\frac{S(n,k-1)S(n,k)}{S^2(n+1,k)}&\ge \frac{(k-1)(n-k+1)}{(k+1)(n-k+2)} -\frac{(n+2)(n-k+1)}{(n+1)(n-k+2)}\\
&=-\frac{(n-k+1) (2 n+k+3)}{(k+1) (n+1) (n-k+2)}
\end{align*}
which is obvious. The proof of Theorem~\ref{S(n-k)-increasing-seq-thm} is complete.
\end{proof}

\section{A conjecture}

Finally, motivated by Theorem~\ref{S(n-k)-increasing-seq-thm}, we pose the following conjecture.

\begin{conj}
For $k,\ell,n\in\mathbb{N}$, let $\mathfrak{S}_{1}(n,k)$ is defined by~\eqref{frak-S(n-k)} and let
\begin{equation}
\mathfrak{S}_{\ell+1}(n,k)=\mathfrak{S}_{\ell}^2(n,k-1) -\mathfrak{S}_{\ell}(n,k-2)\mathfrak{S}_{\ell}(n,k)
\end{equation}
and
\begin{equation}
\mathcal{S}_{\ell}(n,k)=\frac{\mathfrak{S}_{\ell+1}(n,k)}{\mathfrak{S}_{\ell}(n,k)}
\end{equation}
for $n\ge k\ge\ell+2$.
Then the following claims are valid.
\begin{enumerate}
\item
For fixed integers $\ell\in\mathbb{N}$ and $n\ge\ell+3$, the finite sequence $\{\mathfrak{S}_{\ell}(n,k)\}_{\ell+1\le k\le n}$ is logarithmically concave with respect to $k$.
\item
For fixed integers $n\ge k\ge3$, the finite sequence $\{\mathfrak{S}_{\ell}(n,k)\}_{1\le \ell\le k-1}$ is strictly increasing with respect to $\ell$.
\item
For fixed integers $\ell\in\mathbb{N}$ and $n\ge k\ge\ell+1$, the infinite sequence $\{\mathfrak{S}_{\ell}(n+m,k+m)\}_{m\ge0}$ is strictly increasing with respect to $m$.
\item
For fixed integers $\ell\in\mathbb{N}$ and $k\ge\ell+1$, the infinite sequence $\{\mathfrak{S}_{\ell}(n,k)\}_{n\ge k}$ is strictly increasing with respect to $n$.
\item
For fixed integers $n\ge k\ge\ell+2$, the infinite sequence $\{\mathcal{S}_{\ell}(n+m,k+m)\}_{m\ge0}$ is strictly increasing with respect to $m$.
\item
For fixed integers $k\ge\ell+2$, the infinite sequence $\{\mathcal{S}_{\ell}(n,k)\}_{n\ge k}$ is strictly increasing with respect to $n$.
\end{enumerate}
\end{conj}

\begin{rem}
There are some closely related results in the papers~\cite{ANLY-D-12-1238.tex, GJMA-3310-Guo-Qi.tex, Eight-Identy-More.tex, exp-derivative-sum-Combined.tex, Filomat-36-73-1.tex, Bernoulli-Stirling-4P.tex} on Stirling numbers of the first and second kinds $s(n,k)$ and $S(n,k)$.
\end{rem}

\begin{rem}
This paper is a revised version of the preprint~\cite{SirlingTo-Recurr-Ext.tex}.
\end{rem}

\end{document}